\pdfoutput=1
\RequirePackage[l2tabu, orthodox]{nag}
\documentclass[11pt,a4paper,reqno]{amsart}

\usepackage[utf8]{inputenc}
\usepackage[australian]{babel}

\usepackage{microtype}

\usepackage{amsfonts,amsmath,amsthm,amssymb,mathtools}
\usepackage{xcolor}

\usepackage{enumitem}

\usepackage{thmtools, thm-restate}
\usepackage[hidelinks]{hyperref}
\definecolor{darkblue}{rgb}{0.0, 0.0, 0.55}
\hypersetup{%
	colorlinks=true,
	linkcolor = darkblue,
	anchorcolor = darkblue,
	citecolor = darkblue,
	filecolor = darkblue,
	urlcolor = darkblue,
	linktoc=all
}

\usepackage{tikz}

\usepackage[capitalise,noabbrev,sort]{cleveref}

\theoremstyle{plain}
\newtheorem{theorem}{Theorem}[section]
\newtheorem{lemma}[theorem]{Lemma}

\newtheorem{corollary}{Corollary}[theorem]
\newtheorem{definition}[theorem]{Definition}
\theoremstyle{definition}
\newtheorem{example}[theorem]{Example}
\newtheorem{remark}[theorem]{Remark}

\renewcommand\leq\leqslant
\renewcommand\geq\geqslant

\crefalias{theoremx}{theorem}

\begin{document}

\title[On Groups whose cogrowth is the diagonal of a rational series]{On Groups whose cogrowth series is the diagonal of a rational series}

\author{Alex Bishop}
\email{alexbishop1234@gmail.com}
\urladdr{https://alexbishop.github.io}
\address{%
	Section de mathématiques\\
	Université de Genève\\
	rue du Conseil-Général~7-9\\
	1205 Genève, Switzerland}

\begin{abstract}
	We show that if a group contains $\mathbb{Z}^n \times F_m$ as a finite-index subgroup, then its cogrowth series is the diagonal of a rational function for every generating set.
	This answers a question of Pak and Soukup on the cogrowth of virtually abelian groups; and generalises a result by Elder, Rechnitzer, Janse van Rensburg, and Wong on the cogrowth series of the Baumslag-Solitar groups $\mathrm{BS}(N,N)$.
\end{abstract}
\maketitle

\section{Introduction}

The \emph{cogrowth function} $c \colon \mathbb{N}\to \mathbb{N}$ of a finitely-generated group counts the words of a given length which represent the identity.
The asymptotics and complexity of the cogrowth function is of interest in the literature due to its applications to the study of amenability and the word problem.
We note here that in the literature the cogrowth function has also been defined to instead count the freely reduced words that represent the identity.

Let $G$ be a group with finite (monoid) generating set $X$.
In this paper, we study the \emph{cogrowth series} which is simply the generating function of the cogrowth function.
In particular, for each \textit{regular language} $R\subseteq X^*$, we consider the series $C_R(z)$ which counts the words in $R$ which represent the group identity.
Notice then that $C_R(z)$ can be specialised to either of the above definitions of cogrowth.
In particular, we write $C_{X^*}(z)$ for the cogrowth series with respect to all word, and $C_\textsc{Red}(z)$ for the cogrowth series with respect to the (regular) language of freely reduced words.

It is interesting to find correspondences between analytic characterisations of the cogrowth series and classes of groups.
Towards this goal, it was shown in the PhD thesis of Kuksov~\cite{KouksovThesis1998,Kouksov1998} that $C_{\textsc{Red}}(z)$ is a rational series if and only if the group is finite; and it can be shown from a result of Chomsky and Sch\"utzenberger~\cite{Chomsky1963} that each $C_R(z)$ is algebraic for virtually free groups (cf.~\cref{lem:dcf-iff-vfree,lem:det-implies-unambiguous,lem:gfun,lem:intersecion-dcf-regular}).
It is conjectured that $C_\mathrm{Red}(z)$ is algebraic if and only if the group is virtually free.
Furthermore, it is well known that for abelian groups, the series $C_\mathrm{Red}(z)$ is \textit{D-finite} with respect to any finite symmetric generating set (see~\cite[\S5.1]{KouksovThesis1998} or \cite[Theorem~1.2]{Humphries1997}).

In this paper, we are interested in groups for which the cogrowth series $C_R(z)$ can be expressed as the \textit{diagonal} of a (multivariate) rational function as defined in \cref{def:full-diag}.
We note here that the class of such generating functions contains the class of algebraic generating functions (see~\cref{cor:alg-diag-to-rat-diag}), and is contained in the class of \emph{D-finite} generating functions.

Our main objective in this paper is to prove the following.

\begin{restatable*}{theoremx}{ThmA}\label{thm:main}
	If $G$ contains a finite-index subgroup isomorphic to $\mathbb{Z}^n\times F_m$, for some $n,m\geq 0$, then for every finite (monoid) generating set $X$, and every regular language $R\subseteq X^*$, the cogrowth series, $C_R(z)$, can be written as the diagonal of a rational function.
\end{restatable*}

Recall that the \emph{Baumslag-Solitar groups}, $\mathrm{BS}(N,M)$, are defined as
\[
	\mathrm{BS}(N,M)
	=
	\left\langle
	a,t
	\mid
	t a^N t^{-1} = a^M
	\right\rangle
\]
where $N,M \in \mathbb{Z}$.
It is known from \cite{Elder2014} that the cogrowth series $C_{X^*}(z)$ and $C_\textsc{Red}(z)$ of the groups $\mathrm{BS}(N,N)$ are \textit{D-finite} with respect to the standard generating set $X = \{a,a^{-1},t,t^{-1}\}$.
In particular, from the proof of Theorem 4.1 in \cite{Elder2014} we see that the cogrowth series $C_{X^*}(z)$ can be written as the diagonal of an algebraic series in two variables.
Combining this result with \cref{cor:alg-diag-to-rat-diag} (in this paper), we see that $C_{X^*}(z)$ can be written as the diagonal of a rational function in 4 variables.
From \cref{thm:main} (and  \cref{lem:BS_NN}), we may generalise this result to all generating sets as follows.

\begin{restatable*}{corollaryx}{CorAa}\label{cor:main.a}
	The cogrowth series $C_R(z)$ of $\mathrm{BS}(N,N)$, for every $N \in \mathbb{Z}$, is the diagonal of a rational function for every finite (monoid) generating set, and every regular language $R$.
\end{restatable*}

Pak and Soukup~\cite[\S6.5]{Pak2022} asked if the cogrowth series of every finitely-generated virtually abelian group is the diagonal of a rational series for every generating set, and if this can be strengthened to the diagonal of an $\mathbb{N}$-rational series.
Notice that \cref{thm:main} allows us to immediately give an affirmative answer to the first half of this question.
Modifying the proof of \cref{thm:main}, we obtain an answer the second half of this question as follows.

\begin{restatable*}{theoremx}{ThmB}\label{thmB}
	The cogrowth series of a virtually abelian group, $C_R(z)$, is the diagonal of an $\mathbb{N}$-rational series for every finite monoid generating sets and every regular language $R$.
\end{restatable*}

This paper is organised as follows.
In \cref{sec:background}, we introduce our notation and recall the definitions of cogrowth series, \textit{regular languages}, and the \textit{Baumslag-Solitar groups} $\mathrm{BS}(N,N)$.
In \cref{sec:generating-functions}, we recall the definition and properties of algebraic generating functions, and prove \cref{cor:alg-diag-to-rat-diag} which we use in the proof of \cref{thm:main}.
In \cref{sec:formal-languages}, we recall the properties of deterministic and unambiguous context-free languages and provide a proof that the multivariate generating function of unambiguous context-free languages is algebraic.
In \cref{sec:coset-labelled,sec:word-problem}, we reduce the problem of computing the cogrowth series to a problem involving an unambiguous context-free language and a set $\mathcal{Z}\subseteq\mathbb{N}^{p}$ with $\mathbb{N}$-rational generating function.
We then prove our main results in \cref{sec:computing-cogrowth}.

\section{Notation and background}\label{sec:background}

Let $G$ be a group with finite (monoid) generating set $X = \{x_1, x_2,\ldots, x_s\}$,
that contains $H = \mathbb{Z}^n \times F_m$ as a subgroup of finite index $d = [G : H]$.
Then, without loss of generality, we may assume that $H$ is a normal subgroup as otherwise we may apply the following lemma to replace $H$ with a finite-index normal subgroup $H'$ which is isomorphic to $\mathbb{Z}^{n'}\times F_{m'}$ for some $n',m'\in \mathbb{N}$.

\begin{lemma}\label{lem:normal-subgroup}
	If $G$ contains $H = \mathbb{Z}^n\times F_m$ as a finite-index subgroup.
	Then $G$ contains a finite-index normal subgroup $H'\cong \mathbb{Z}^{n'} \times F_{m'}$ where $n',m'\in \mathbb{N}$.
\end{lemma}

\begin{proof}
We begin by taking the normal core of $H$ as
\[
	H' \coloneqq\mathrm{core}_G(H)
	=
	\bigcap_{g\in G} g H g^{-1}.
\]
It is well known that $H'$ is a finite-index normal subgroup of both $G$ and $H$. 
Since $H'$ is a subgroup of $H=\mathbb{Z}^n \times F_n$, it follows from Proposition 1.5 and Corollary 1.7 in \cite{Delgado2013} that $H' \cong \mathbb{Z}^{n'} \times F$ where $0\leq n' \leq n+1$ and $F$ is a free group.
Since $H'$ is finite index in the finitely generated group $H$, we see that $H'$ and thus $F$ are finitely generated.
We may then conclude that $H' \cong \mathbb{Z}^{n'} \times F_{m'}$ for some $n',m'\in \mathbb{N}$.
\end{proof}

Moreover, for the proof of \cref{thmB}, we will also require the following result for the virtually abelian case.

\begin{lemma}\label{lem:normal-abelian}
	If $G$ is a finitely-generated virtually abelian group, then $G$ contains a finite-index normal subgroup $H$ which is isomorphic to $\mathbb{Z}^n$ for some $n \in \mathbb{N}$.
\end{lemma}

\begin{proof}
	By definition, $G$ must contain an abelian subgroup $A$ of finite index $d = [G:A]$, and thus $A$ is finitely generated.
	From the classification of finitely-generated abelian groups, we find that
	$
		A = F \times \mathbb{Z}^n
	$
	where $F$ is a finite group and $n \in \mathbb{N}$.
	We then see that $G$ contains $\mathbb{Z}^n$ as a subgroup of index $|F| \cdot d$.
	
	We now take the normal core of $\mathbb{Z}^n$ in $G$ as
	\[
		H \coloneqq\mathrm{core}_G(\mathbb{Z}^n)
		=
		\bigcap_{g\in G} g \mathbb{Z}^n g^{-1}
	\]
	which is a finite-index normal subgroup of both $G$ and $\mathbb{Z}^n$.
	
	From pp.~100-1 in \cite{Robinson1996}, we see that since $H$ is a finite-index subgroup of $\mathbb{Z}^n$, then $H \cong \mathbb{Z}^n$ as required.
\end{proof}

Fix a finite set $T = \{t_1 = 1,t_2,\ldots,t_d\} \subset G$ of (right) coset representatives of the normal subgroup $H = \mathbb{Z}^n\times F_m$ in $G$.
Notice then that for each element $g \in G$, there is a unique choice of $h \in H$ and $t \in T$ such that $g = h \cdot t$.
We define the map $\rho\colon G \to T$ such that $\rho(g) = t_i$ for each $g \in H \cdot t_i$.

For each word $w \in X^*$, we write $\overline{w}$ to denote the corresponding element of $G$.
Moreover, we write $|w|_X$ for the word length of $w \in X^*$, and $\ell_X(g)$ for the length of an element $g\in G$ with respect to $X$, i.e.,
\[
	\ell_X(g)
	=
	\min\{
		|w|_X \in \mathbb{N}
	\mid
		w\in X^*\text{ with }\overline{w} = g
	\}.
\]
We now recall the definition of a \emph{rational generating functions} as follows.

\subsection{Rational generating functions}\label{sec:rational-series}

We say that a multivariate power series $f(\mathbf{x}) \in \mathbb{Q}[[\mathbf{x}]]$ in variables $\mathbf{x} = (x_1,x_2,\ldots,x_k)$ is  \emph{rational} if there are polynomials $P(\mathbf{x}),Q(\mathbf{x}) \in \mathbb{Z}[\mathbf{x}]$ such that
$
	f(\mathbf{x})
	=
	P(\mathbf{x})/Q(\mathbf{x})
$.

In this paper, we also make use of a subclass of rational series known as the \emph{$\mathbb{N}$-rational} series.
In particular, the class of \emph{$\mathbb{N}$-rational} series in variables $\mathbf{x} = (x_1,x_2,\ldots,x_k)$ is the smallest subset of $\mathbb{N}[[\mathbf{x}]]$ which contains the polynomials $\mathbb{N}[\mathbf{x}]$
and is closed under addition, multiplication, and \textit{quasi-inverse},
that is, if
$f(\mathbf{x})$, $g(\mathbf{x})$ and $h(\mathbf{x})$ are $\mathbb{N}$-rational with $h(0,\ldots.,0)=0$, then so are
$f(\mathbf{x}) + g(\mathbf{x})$,
$f(\mathbf{x}) g(\mathbf{x})$ and
$1/(1-h(\mathbf{x}))$.

We now recall the definitions of \emph{regular languages} and show that their multivariate generating functions are $\mathbb{N}$-rational.

\subsection{Regular language}\label{sec:regular-language}

An \emph{alphabet} is a finite set of symbols $\Sigma$, and a \emph{language} is a subset $L\subseteq \Sigma^*$ of words in the letters of $\Sigma$.
We define the family of regular languages from the class of \emph{finite-state automata} as defined below.

A \emph{finite-state automaton} is a tuple $\mathcal{M} = (\Sigma,Q,\delta,q_\mathrm{start},A)$ where
\begin{itemize}
	\item $\Sigma$ is an alphabet;
	\item $Q$ is a finite set of \emph{states};
	\item $\delta\colon Q \times\Sigma \to  Q$ is the \emph{transition function};
	\item $q_\mathrm{start} \in Q$ is the start state; and
	\item $A \subseteq Q$ is the set of accepting states.
\end{itemize}
From this description, we construct a finite directed edge-labelled graph with vertex set $Q$, and an edge $q\to^a q'$ for each $\delta(q,a)=q'$.
Notice that for each word $w = w_1 w_2 \cdots w_k \in \Sigma^*$, there is exactly one path labelled as
\[
	q_\mathrm{start}
	\xrightarrow{w_1}
	q_1
	\xrightarrow{w_2}
	q_2
	\xrightarrow{w_3}
	\cdots
	\xrightarrow{w_k}
	q_k.
\]
In particular, we see that $p_{i+1} = \delta(p_{i},w_i)$ for each $i \in \{0,1,\ldots,k\}$ where we define $q_0 = q_\mathrm{start}$.
We then say that such a path is \emph{accepting} if the final state $q_k \in A$.
From this, we define the class of regular languages as follows.

\begin{definition}\label{def:regular-lang}
	A language $R \subseteq \Sigma^*$ is \emph{regular}, if there exists a finite-state automaton which accepts a word if and only if it belongs to $R$.
\end{definition}

For example, if $X$ is a finite generating set for a group, then both $X^*$ and the set of freely reduced words in $X^*$ are regular languages.
Regular languages satisfy the following well-known closure properties.

\begin{lemma}[Theorems 4.4, 4.8, 4.14 and 4.16 in \cite{Hopcroft2006}]\label{lem:reg-lang-closure}
	The family of regular language is closed under intersection, union, monoid homomorphism, and inverse monoid homomorphism.
\end{lemma}

It is also well known that the multivariate generating function of a regular language is $\mathbb{N}$-rational.
We provide a short proof of this as follows.

\begin{lemma}\label{lem:regular_are_N-rational}
	Let $R\subseteq\Sigma^*$ be a regular language with $\Sigma = \{a_1,\ldots,a_{n}\}$, and define the multivariate generating function of $R$ as
	\[
		f(x_1,\ldots,x_{n},z)
		=
		\sum_{k_1,\ldots,k_{n}\in \mathbb{N}}
		c(k_1,\ldots,k_{n})
		x_1^{k_1}\cdots x_{n}^{k_n} z^{k_1+\cdots+k_{n}}
	\]
	where $c(k_1,\ldots,k_{n})$ counts the words in $R$ which, for each $i$, contain exactly $k_i$ instances of the letter $a_i$.
	Then, $f(x_1,\ldots,x_{n},z)$ is $\mathbb{N}$-rational.
\end{lemma}

\begin{proof}

Let $R\subseteq \Sigma^*$ be a regular language, and let $\mathcal{M} = (\Sigma,Q,\delta,q_\mathrm{start},A)$ be a finite-state automaton which accepts $R$.

For each state $q\in Q$, we define $F_q(\mathbf{x},z)$ to be the generating function of the language
\[
	L_q
	=\left\{
		w = w_1 w_2 \cdots w_p \in \Sigma^*
	\ \middle|\ 
		q \xrightarrow{w_1}
		q_1\xrightarrow{w_2}
		\cdots
		\xrightarrow{w_p} q_p
	\text{ where }
		q_p \in A
	\right\}.
\]
We notice that $f(\mathbf{x},z) = F_{q_\mathrm{start}}(\mathbf{x},z)$.

We then see that for each $q \in Q$, we have the recurrence
\begin{align*}
	F_{q}(\mathbf{x},z)
	&=
	1+
	\sum_{i = 1}^{|\Sigma|}
	x_i z\,
	F_{\delta(q,a_i)}(\mathbf{x},z)
	&&\text{if }q \in A;\text{ and}
	\\
	F_{q}(\mathbf{x},z)
	&=
	\sum_{i = 1}^{|\Sigma|}
	x_i z\,
	F_{\delta(q,a_i)}(\mathbf{x},z)
	&&\text{if }q \notin A.
\end{align*}

Solving these recurrences by substitution, we find that each $F_{q}(\mathbf{x},z)$ is $\mathbb{N}$-rational and thus $f(\mathbf{x},z) = F_\mathrm{start}(\mathbf{x},z)$ is $\mathbb{N}$-rational.
\end{proof}

\subsection{Cogrowth series}

Let $R\subseteq X^*$ be a regular language.
We define the \emph{cogrowth series of $G$ with respect to $R$} as
\begin{equation}\label{eq:cogrowth}
	C_{R}(z)
	\coloneqq
	\sum_{k=0}^\infty c_n z^n \in \mathbb{N}[[z]]
\end{equation}
where each coefficient
\[
	c_n
	=
	\#
	\{
		w\in R\subseteq X^*
	\mid
		|w|_X = n\text{ and } \overline{w}=1
	\}
\]
counts the words of length $n$ which represent the group identity.

In this paper, we show that the cogrowth series of certain groups can be written as the diagonal of rational series.
Thus, we define the \emph{diagonal} of a multivariate power series as follows.

\subsection{Formal power series and diagonals}

We first define the \emph{primitive diagonals} of $f(\mathbf{x})$ in variables $\mathbf{x}= (x_1,\ldots,x_p)$ as follows.
This definition is used in the statement and proof of \cref{lem:algebraic-diag-rat,cor:alg-diag-to-rat-diag}.

\begin{definition}\label{def:partial-diag}
	Let $f(\mathbf{x}) \in \mathbb{Q}[[\mathbf{x}]]$ be a multivariate power series in variables $\mathbf{x} = (x_1,x_2,\ldots,x_p)$,
	then the \emph{primitive diagonal} $I_{x_1,x_2} f$ is defined such that
	\[
		I_{x_1,x_2} f
		(x_1,x_3,x_4,\ldots,x_p)
		=
		\sum_{k_1,k_3,\ldots,k_p,\in \mathbb{N}}
		a(k_1,k_1,k_3,\ldots,k_p)
		x_1^{k_1} x_3^{k_3}\cdots x_p^{k_p}.
	\]
	That is, we select the terms where the exponent of $x_1$ and $x_2$ are equal.
	Notice that $I_{x_1,x_2} f$ is a power series over the set of variables $\mathbf{x}\setminus\{x_2\}$.
	
	For each $i<j$, we can analogously define the primitive diagonal $I_{x_i, x_j} f$ as a generating function over the variables $\mathbf{x}\setminus\{x_j\}$.
\end{definition}

We may then define the \emph{(complete) diagonal} of a power series as follows.

\begin{definition}\label{def:full-diag}
	Let $f(\mathbf{x})\in \mathbb{Q}[[\mathbf{x}]]$ be a multivariate power series in the variables $\mathbf{x} = (x_1,x_2,\ldots,x_p)$, then the \emph{(complete) diagonal} of $f(\mathbf{x})$ is
	\[
		\mathrm{Diag}(f)(z)
		=
		\left(I_{x_1,x_2} I_{x_1,x_3} \cdots I_{x_1, x_p} f\right)(z)
		=
		\sum_{k\in \mathbb{N}}
		a(k,k,\ldots,k)
		\, z^k.
	\]
	Notice that the complete diagonal is a univariate power series in $\mathbb{Q}[[z]]$.
\end{definition}

In this paper, we are interested in showing that the cogrowth series can be written as the complete diagonal of a rational series as in \cref{sec:rational-series}.

\subsection{Baumslag-Solitar groups}

An application of the results in this paper is to the Baumslag-Solitar groups of the form $\mathrm{BS}(N,N)$.
Towards this result, we provide the following characterisation.

\begin{lemma}\label{lem:BS_NN}
	For each $N \geq 1$, $\mathrm{BS}(N,N)$ is virtually $\mathbb{Z} \times F_N$.
\end{lemma}

\begin{proof}
	Let
	\[
		\mathrm{BS}(N,N)
		=
		\left\langle
			a,t
		\mid
			t a^N t^{-1} = a^N
		\right\rangle.
	\]
	Then, we see that $\mathrm{BS}(N,N)$ contains a free subgroup generated as
	\[
		F_N
		=
		\left\langle
			t,\,
			a t a^{-1},\,
			a^2 t a^{-2},\,
			a^3 t a^{-3},\,\ldots,\,
			a^{N-1} t a^{1-N}
		\right\rangle.
	\]
	Moreover, we see that the centre of $\mathrm{BS}(N,N)$ is a cyclic group generated as
	\[
		\mathbb{Z} = \left\langle a^N \right\rangle.
	\]
	Thus,
	\[
		\mathbb{Z} \times F_N
		=
		\left\langle a^N \right\rangle
		\times
		\left\langle
			t,\,
			a t a^{-1},\,
			a^2 t a^{-2},\,
			a^3 t a^{-3},\,\ldots,\,
			a^{N-1} t a^{1-N}
		\right\rangle
	\]
	is a subgroup of $\mathrm{BS}(N,N)$.
	We then see that
	\[
		\mathrm{BS}(N,N)
		=
		\left(\mathbb{Z} \times F_N\right)
		\cdot
		\{
			1, a, a^2, \ldots, a^{N-1}
		\}
	\]
	and thus, $\mathrm{BS}(N,N)$ contains $\mathbb{Z} \times F_N$ as a subgroup of index $N$.
\end{proof}

\section{Algebraic Generating Functions}\label{sec:generating-functions}

In \cref{thm:main}, we show that the generating function for a certain class of groups can be written as the diagonal of a rational series.
We accomplish this by first showing that the generating function can be written as the diagonal of an algebraic series, then using a result of Denef and Lipshitz~\cite{Denef1987}, we strengthen this to the diagonal of a rational series.
Thus, we begin by recalling the definition of an algebraic generating function as follows.
In this section, we write $\mathbf{x}$ for a finite number of variables $\mathbf{x} = (x_1,x_2,\ldots,x_p)$.

\begin{definition}
	A generating function $a(\mathbf{x}) \in \mathbb{Z}[[\mathbf{x}]]$ is called \emph{algebraic} if there is a non-trivial polynomial $P(\mathbf{x},z) \in \mathbb{Z}[\mathbf{x},z]$ such that $P(\mathbf{x},a(\mathbf{x})) = 0$. 
\end{definition}

It is well known that the class of algebraic series in $\mathbb{Z}[[\mathbf{x}]]$ form a ring and that rational generating functions are a subset of the algebraic ones.
We provide a short proof of these properties as follows.

\begin{lemma}\label{lem:alg-closure}
	Every rational generating function in $\mathbb{Z}[[\mathbf{x}]]$ is algebraic, and algebraic generating functions are closed under addition and multiplication.
\end{lemma}

\begin{proof}
	Suppose that $f(\mathbf{x}) = a(\mathbf{x})/b(\mathbf{x})$ with $a(\mathbf{x}),b(\mathbf{x})\in \mathbb{Z}[\mathbf{x}]$, then $f(\mathbf{x})$ is algebraic since $P(\mathbf{x},f(\mathbf{x})) = 0$ where $P(\mathbf{x},z) = a(\mathbf{x}) - b(\mathbf{x}) z$.
	
	Now suppose that $\alpha(\mathbf{x}),\beta(\mathbf{x}) \in \mathbb{Z}[[\mathbf{x}]]$ are both algebraic, then there are square matrices $A = (a_{i,j})$ and $B=(b_{i,j})$, where each $a_{i,j},b_{i,j} \in\mathbb{Z}[[\mathbf{x}]]$, such that $\alpha(\mathbf{x})$ and $\beta(\mathbf{x})$ are eigenvalues of $A$ and $B$ with eigenvectors $v_\alpha$ and $v_\beta$, respectively.
	In particular, $A$ and $B$ are the companion matrices of the polynomials which witness  $\alpha(\mathbf{x})$ and $\beta(\mathbf{x})$ being algebraic.
	We then see that $\alpha(\mathbf{x})\cdot\beta(\mathbf{x})$ and $\alpha(\mathbf{x})+\beta(\mathbf{x})$ are eigenvalues of the matrices
	$A\otimes B$
	and
	$(A\otimes I_B) + (I_A \otimes B)$, respectively, both with eigenvector $v_\alpha \otimes v_\beta$
	where $\otimes$ is the tensor product, and $I_A$ and $I_B$ are the identity matrices with the same dimensions as $A$ and $B$, respectively.
	We then see that $\alpha(\mathbf{x})\cdot\beta(\mathbf{x})$ and $\alpha(\mathbf{x})+\beta(\mathbf{x})$ are algebraic as witnessed by the polynomials
	\begin{align*}
		P(\mathbf{x},z)&=\det\left(A\otimes B - z\, I_{A\otimes B}\right)
		\quad\text{and }\\
		Q(\mathbf{x},z)&=\det\left((A\otimes I_B) + (I_A \otimes B) - z\, I_{(A\otimes I_B) + (I_A \otimes B)}\right),
	\end{align*}respectively.
\end{proof}

Denef and Lipshitz~\cite{Denef1987} proved a following lemma for all algebraic series whose coefficients belong to an \textit{excellent local integral domain} of which any field is an example.
That is, the following is one case of Theorem 6.2 in \cite{Denef1987}.

\begin{lemma}[Theorem 6.2 in \cite{Denef1987}]\label{lem:algebraic-diag-rat}
	Let $a(\mathbf{x}) \in  \mathbb{Z}[\mathbf{x}]]$ be an algebraic series.
	Then, there is a rational series $r(\mathbf{x},\mathbf{y})\in \mathbb{Q}[[\mathbf{x},\mathbf{y}]]$ such that 
	\[
		a(\mathbf{x}) = I_{x_1,y_1}I_{x_2,y_2}\cdots I_{x_p,y_p} \left(r(\mathbf{x},\mathbf{y})\right)
	\]
	where $\mathbf{x} = (x_1,x_2,\ldots,x_p)$ and $\mathbf{y} = (y_1,y_2,\ldots,y_p)$.
\end{lemma}

From the above lemma, we immediately show the following corollary.

\begin{corollary}\label{cor:alg-diag-to-rat-diag}
	Suppose that $f(z) \in \mathbb{Z}[[z]]$ is the diagonal of an algebraic power series $a(\mathbf{x}) \in \mathbb{Z}[[\mathbf{x}]]$ in $p$ variables $\mathbf{x} = (x_1,x_2,\ldots,x_p)$, that is,
	\[
		f(z) = \mathrm{Diag}(a)(z).
	\]
	Then, there is a rational series $r(\mathbf{x},\mathbf{y}) \in \mathbb{Q}[[\mathbf{x},\mathbf{y}]]$ such that
	\[
		f(z) = \mathrm{Diag}
		\left(r(\mathbf{x},\mathbf{y})\right)
		(z)
	\]
	where $\mathbf{x} = (x_1,x_2,\ldots,x_p)$ and $\mathbf{y} = (y_1,y_2,\ldots,y_p)$.
\end{corollary}

\begin{proof}
	Suppose that $f(z)$ can be written as the diagonal of an algebraic series $a(\mathbf{x}) \in \mathbb{Z}[[\mathbf{x}]]$, that is,
	$
		f(z)
		=
		\mathrm{Diag}(a)(z)
	$.
	From \cref{lem:algebraic-diag-rat}, there is a rational series $r(\mathbf{x},\mathbf{y})\in \mathbb{Q}[[\mathbf{x},\mathbf{y}]]$ such that 
	\[
	a(\mathbf{x}) = I_{x_1,y_1}I_{x_2,y_2}\cdots I_{x_p,y_p} \left(r(\mathbf{x},\mathbf{y})\right)
	\]
	From this, we see that
	\[
		f(z)
		=
		\mathrm{Diag}(a)(z)
		=
		\mathrm{Diag}
		\left(r(\mathbf{x},\mathbf{y})\right)
		(z)
	\]
	as desired.
\end{proof}

\section{Context-free languages}\label{sec:formal-languages}

In this section, we recall the definition of \emph{context-free} languages and their connections with algebraic generating functions and virtually free groups.
We begin by defining \emph{context-free grammars} as follows.
We will use this definition in the proof of \cref{lem:gfun}.

A \emph{context-free grammar} is a tuple $\mathcal{G} = (\Sigma, \Gamma, S, \Delta)$ where
\begin{itemize}
	\item $\Sigma$ is the alphabet of \emph{terminals};
	\item $\Gamma$ is an alphabet of \emph{nonterminals} which is disjoint from $\Sigma$;
	\item $S \in \Gamma$ is the \emph{starting symbol}; and
	\item $\Delta \in \Gamma \times (\Sigma\cup\Gamma)^*$ is a finite set of \emph{replacement rules}. 
\end{itemize}
The replacement rules $(A,u)\in \Delta$ allow us to replace instances of a nonterminal symbols $A \in \Gamma$ with words $u \in (\Sigma\cup\Gamma)^*$.
The context-free grammar $\mathcal{G}$ is said to \emph{produce} a word $w \in \Sigma^*$ if, starting with the word which contains only the letter $S$,  we may apply a sequence of replacement rules from $\Delta$ and obtain the word $w$. 
We write $L(\mathcal{G}) \subseteq \Sigma^*$ for the language of all words which are produced by the grammar $\mathcal{G}$.

\begin{example}\label{ex:grammar}
	Let $F_2 = \left\langle a,b \mid-\right\rangle$ be the rank-2 free group with generating set $\Sigma = \{a,a^{-1},b,b^{-1}\}$.
	Then, the language of words $w \in \Sigma^*$ which correspond to the trivial element is produced by the context-free grammar $\mathcal{G} = (\Sigma, \Gamma, S, \Delta)$ with $\Gamma = \{S\}$ and replacements
	\[
		\Delta = \left\{
			(S, \varepsilon),\,
			(S, aSa^{-1}S),\,
			(S, bSb^{-1}S),\,
			(S, a^{-1}SaS),\,
			(S, b^{-1}SbS)
		\right\}.
	\]
	Thus, the language of all such words is context-free.
\end{example}

A \emph{derivation tree} is a rooted tree with ordered children which completely describes how a context-free grammar produces a particular word.
In particular, each node in the tree has children that are determined by a replacement rule.
For example, the word $w = aba^{-1}ab^{-1}a^{-1}bb^{-1}$ is produced by the grammar given in \cref{ex:grammar}, and a derivation tree of $w$ is given in \cref{fig:derivation-tree}.

\begin{figure}[hpt]
\centering
\begin{tikzpicture}[
		level distance=3em,
		level 1/.style={sibling distance=6em},
		level 2/.style={sibling distance=3em},
	]
	\node {$S$}
	child {node {$a$}}
	child {node {$S$}
		child{node {$b$}}
		child{node {$S$}
			child{node {$a^{-1}$}}
			child{node {$S$} child{node {$\varepsilon$}} }
			child{node {$a$}}
			child{node {$S$} child{node {$\varepsilon$}} }
		}
		child{node {$b^{-1}$}}
		child{node {$S$} child{node {$\varepsilon$}} }
	}
	child {node {$a^{-1}$}}
	child {node {$S$}
		child{node {$b$}}
		child{node {$S$} child{node {$\varepsilon$}} }
		child{node {$b^{-1}$}}
		child{node {$S$} child{node {$\varepsilon$}} }
	}
	;
\end{tikzpicture}
\caption{Derivation tree for $w = aba^{-1}ab^{-1}a^{-1}bb^{-1}$ in \cref{ex:grammar}.}\label{fig:derivation-tree}
\end{figure}
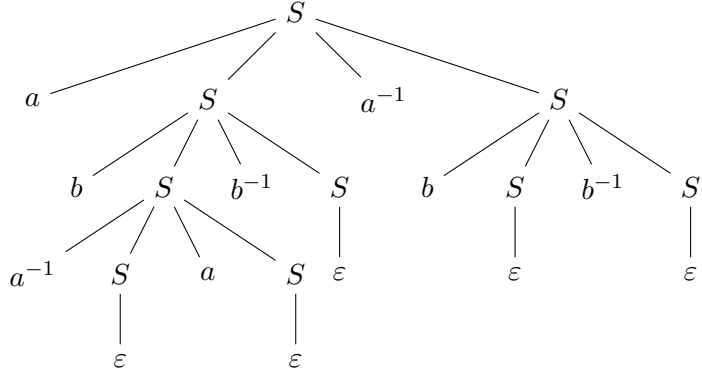

We then say that a context-free grammar $\mathcal{G}$ is \emph{unambiguous} if each word $w\in L(\mathcal{G})$ has exactly one associated derivation tree.
In this section, we will often refer to a subclass (cf.~\cref{lem:det-implies-unambiguous}) of unambiguous context-free languages known a \emph{deterministic} context-free.
For brevity, we do not define the class of deterministic context-free language in this paper.

\begin{lemma}[Theorem 13.15 in \cite{rich2008automata}]\label{lem:det-implies-unambiguous}
	Every deterministic context-free language is unambiguous context-free.
\end{lemma}

We then have the following well-known characterisation for the generating function of unambiguous (and thus deterministic) context-free languages.
This result was proven for noncommutative generating functions by Chomsky and Sch\"utzenberger in \cite{Chomsky1963}.
We provide a short proof for the multivariate (commutative) generating function.

\begin{lemma}\label{lem:gfun}
	Let $L \subseteq \Sigma^*$ be an unambiguous context-free language over the alphabet $\Sigma = \{a_1,a_2,\ldots,a_{n}\}$ with generating function
	\[
		f(x_1,\ldots,x_{n}, z)
		=
		\sum_{k_1,\ldots,k_{n}\in \mathbb{N}}
		c(k_1,\ldots,k_{n})\,
		x_1^{k_1}
		\cdots
		x_{n}^{k_{n}}
		z^{k_1 + \cdots +k_{n}}
	\]
	where each $c(k_1,k_2,\ldots,k_{n})$ counts the words in $L$ that, for each $i$, contain exactly $k_i$ instances of the letter $a_i$.
	Then, $f(x_1,\ldots,x_{n}, z)$ is algebraic.
\end{lemma}

\begin{proof}
Let $\mathcal{G} = (\Sigma, \Gamma, S, \Delta)$ be an unambiguous context-free grammar such that $L = L(\mathcal{G})$.
For each nonterminal $A \in \Gamma$, we write $F_A(\mathbf{x},z)$ for the generating function which counts the words in $\Sigma^*$ that can be produced from the nonterminal $A$ using  a sequence replacement rules in $\Delta$.

We then see that for each $A \in \Gamma$, we have the recurrence formula
\[
	F_A(\mathbf{x},z)
	=
	\sum_{(A, w) \in \Delta}
		\left(
			\prod_{i=1}^n
				(x_iz)^{|w|_{a_i}}
		\right)
		\left(
			\prod_{B \in \Gamma}
				F_B(\mathbf{x},z)^{|w|_B}
		\right)
\]
where each $|w|_{a_i}$ and $|w|_B$ counts the number of instances of the letters $a_i \in \Sigma$ and $B \in \Gamma$ in the word $w \in (\Sigma\cup\Gamma)^*$.
Notice that each term in the above sum corresponds to a possible choice of children of a node labelled with the non-terminal letter $A$ in a derivation tree.

Thus, we have a system of $|\Gamma|$ polynomial equations in the variables $(F_{A}(\mathbf{x},z))_{A\in \Gamma}$ with coefficients in $\mathbb{Z}[\mathbf{x},z]$.
It is then known that such a system of equation has algebraic solutions.
In particular, we may perform a sequence algebraic eliminations (see~Appendix B.1 in \cite{flajolet2009}) to find a polynomial equation that witnesses $f(\mathbf{x},z)=F_S(\mathbf{x},z)$ being algebraic.
\end{proof}

It is a well-known result that the class of deterministic context-free languages classify the \textit{word problem} of virtually free groups as in the following lemma.
This result is used in the proof of \cref{lem:det-cf}.

\begin{lemma}[Theorem in \cite{muller1983groups} and Corollary 2.10 in \cite{muller1985theory}]\label{lem:dcf-iff-vfree}
	Let $G$ be a group with finite symmetric generating set $S$, then the word problem
	\[
		\mathrm{WP}_S
		\coloneqq
		\{
			w\in S^*
		\mid
			\overline{w}=1
		\}
	\]
	is deterministic context-free if and only if $G$ is virtually free.
\end{lemma}

We have the following closure properties for the family of deterministic context-free languages which we use in the proof of \cref{lem:det-cf}.

\begin{lemma}[Lemma 3 on p.~128 of \cite{hoogeboom2004}]\label{lem:intersecion-dcf-regular}
	The family of deterministic context-free languages is closed under intersection with regular language.
	Moreover, if $p \colon \Sigma^* \to X^*$ is a monoid homomorphism, and $L \subseteq X^*$ is deterministic context-free, then the language $p^{-1}(L) \subseteq \Sigma^*$ is also deterministic context-free.
\end{lemma}

\section{Coset-labelled words}\label{sec:coset-labelled}

In this section, we describe a method of decorating words in $X^*$ with cosets representatives in $T$.
These decorated words are used in the construction of a deterministic context-free language in \cref{lem:det-cf} which is then used in the proof of \cref{thm:main}.

Let $w = x_{k_1} x_{k_2} \cdots x_{k_\ell} \in X^*$ be a word of length $\ell$, then for each prefix $x_{k_1} x_{k_2} \cdots x_{k_{p}}$ of $w$, we define $t_{c_{p+1}} = \rho(x_{k_1} x_{k_2} \cdots x_{k_{p}}) \in T$.
In particular, we see that $t_{c_1} = 1$ and $t_{c_{\ell+1}} = \rho({w})$.
Recall here that $\rho\colon G \to T$ is the map for which $g \in H \cdot \rho(g)$ for each $g\in G$.
Notice then that,
\[
	\overline{w}
	=
	(
		t_{c_1}
		x_{k_1}
		t_{c_2}^{-1}
	)
	(
		t_{c_2}
		x_{k_2}
		t_{c_3}^{-1}
	)
	\cdots
	(
		t_{c_\ell}
		x_{k_\ell}
		t_{c_{\ell+1}}^{-1}
	)
	\cdot t_{c_{\ell+1}}
\]
where each factor $t_{c_i} x_{k_i} t_{c_{i+1}}^{-1}$ corresponds to an element in the finite-index normal subgroup $H$.
We see that $t_{c_{i+1}} = \rho(t_{c_i} x_{k_i})$ for each $i$, and thus
\begin{multline}\label{eq:factor-word}
	\overline{w}
	=
	\left(
		t_{c_1}
		x_{k_1}
		(\rho(t_{c_1}x_{k_1}))^{-1}
	\right)
	\left(
		t_{c_2}
		x_{k_2}
		(\rho(t_{c_2}x_{k_2}))^{-1}
	\right)
	\\\cdots
	\left(
		t_{c_\ell}
		x_{k_\ell}
		(\rho(t_{c_\ell}x_{k_\ell}))^{-1}
	\right)
	\cdot
	\rho(t_{c_\ell}x_{k_\ell})
\end{multline}
where each factor $t_{c_i} x_{k_i} (\rho(t_{c_i}x_{k_i}))^{-1}$ corresponds to an element in $H$.
Notice also that if $\overline{w} \in H$, then $\rho(t_{c_\ell} x_{k_\ell}) = 1$, and thus \cref{eq:factor-word} gives us an expression for $\overline{w}$ factored into $\ell$ parts, each belonging to $H$.

For each generator $x_i \in X$, and each coset representative $t_j \in T$, we introduce a new symbol $\sigma_{i,j}$ which corresponds to the word $t_j x_{i} (\rho(t_{j}x_{i}))^{-1}$ as in (\ref{eq:factor-word}).
We then write $\Sigma$ for the finite set of all such symbols
\begin{equation}\label{eq:sigma}
	\Sigma
	=
	\{
		\sigma_{i,j}
	\mid
		x_i \in X,\  
		t_j \in T
	\}.
\end{equation}
Notice here that $\Sigma$ contains exactly $sd = |X| \cdot |T|$ symbols.
We then define a length-preserving map $\varphi\colon X^* \to \Sigma^*$ as follows.

\begin{definition}\label{def:coset-labelled}
	Let $\varphi \colon X^* \to \Sigma^*$ be the map defined such that for each word
	$
		w = x_{k_1} x_{k_2} \cdots x_{k_\ell} \in X^*,
	$
	we have
	$
		\varphi(w)
		=
		\sigma_{k_1,c_1}
		\sigma_{k_2,c_2}
		\cdots
		\sigma_{k_\ell,c_\ell}
		\in \Sigma^*
	$
	where each $c_i$ is as in \cref{eq:factor-word}.
	We call the words $\varphi(X^*)$ \emph{coset labelled}.
\end{definition}

From the above definition, it is clear that $\varphi\colon X^* \to \Sigma^*$ is injective as it has a left inverse $\mu\colon \Sigma^* \to X^*$ defined as follows.

\begin{definition}
	Let $\mu\colon \Sigma^* \to X^*$ be the monoid homomorphism defined such that $\mu(\sigma_{i,j}) = x_i$ for each $\sigma_{i,j} \in \Sigma$.
	We then see that $(\mu \circ \varphi)(w) = w$ for each $w \in X^*$ as the map $\varphi\colon X^*\to\Sigma^*$ only decorates to the letters of $w$ with cosets representatives from $T$, and $\mu\colon \Sigma^*\to X^*$ removes these decorations.
\end{definition}

For each regular language $R\in X^*$, we write $\mathcal{L}_R \subseteq \Sigma^*$ for the set of words in $\varphi(R)$ which correspond to an element in the trivial coset as follows.

\begin{lemma}\label{def:L_R}
	Let $R \subseteq X^*$ be the language defined as
	\[
		\mathcal{L}_R
		=
		\{
			\varphi(w)\in \Sigma^*
		\mid
			w \in R
			\ \mathrm{with}\ 
			\overline{w} \in H
		\}.
	\]
	Then, $\mathcal{L}_R$ is a regular language.
\end{lemma}

\begin{proof}
We first notice that $\mathcal{L}_R$ can be written as the intersection
\[
	\mathcal{L}_R = \Phi_X \cap \mu^{-1}(R)
	\quad\text{where}\quad
	\Phi_X = \varphi(\{w\in X^*\mid \overline{w}\in H\}).
\]
From \cref{lem:reg-lang-closure}, we know that the family of regular languages is closed under intersection and inverse monoid homomorphism.
Thus, $\mu^{-1}(R)$ is a regular language, and it only remains to be shown that $\Phi_X$ is regular.

We may construct a finite-state automaton for $\Phi_X\subseteq \Sigma^*$ with
\begin{itemize}
	\item a state $q_{t_i} \in Q$ for each coset representative $t_i\in T$;
	\item a transition $q_{t_j} \to^{\sigma_{i,j}} q_{t_{j'}}$ for each $\sigma_{i,j}\in \Sigma$ where $t_{j'} = \rho(t_j x_i)$; where
	\item $q_1$ is both the starting state and the only accepting state.
\end{itemize}
Thus, $\Phi_X$ is a regular language as required.
\end{proof}

Since $\varphi$ is injective and length preserving, we see that $\varphi$ gives us a length-preserving bijection from the set of words $\{w \in R \mid \overline{w} \in H\}\in X^*$ to the set of words in $\mathcal{L}_R\subseteq \Sigma^*$.
Thus, we require a method to check membership for $\mathcal{L}_R$, and to verify if $\overline{\mu(w)}=1$ for a given word $w \in \mathcal{L}_R$.

\section{On words which represent the identity}\label{sec:word-problem}

The aim of this section is to prove \cref{lem:det-cf,lem:equation}.
In \cref{thm:main}, we combine these lemmas to show that the series $C_R(z)$ can be written as a diagonal of an algebraic power series.
After which we use \cref{cor:alg-diag-to-rat-diag} to strengthen this to the diagonal of a rational series.

Let $w = x_{k_1}x_{k_2}\cdots x_{k_\ell} \in X^*$ with
$
	\varphi(w) =
	\sigma_{k_1,c_1}
	\sigma_{k_2,c_2}
	\cdots
	\sigma_{k_\ell,c_\ell}
	\in \mathcal{L}_R,
$
then
\[
	\overline{w}
	=
	\left(
	\overline{
		t_{c_1}
		x_{k_1}
		(\rho(t_{c_1}x_{k_1}))^{-1}
	}
	\right)
	\left(
	\overline{
		t_{c_2}
		x_{k_2}
		(\rho(t_{c_2}x_{k_2}))^{-1}
	}
	\right)
	\\\cdots
	\left(
	\overline{
		t_{c_\ell}
		x_{k_\ell}
		(\rho(t_{c_\ell}x_{k_\ell}))^{-1}
	}
	\right)
\]
where each factor represents an element in the subgroup $H$, that is, each
\[
	\overline{
		t_{j}
		x_{i}
		(\rho(t_{j} x_{i}))^{-1}
	}
	\in H.
\]
We define two maps $\pi_{\mathbb{Z}^n}\colon \Sigma^* \to \mathbb{Z}^n$ and $\pi_{F_m}\colon \Sigma^* \to F_m$ as follows

\begin{definition}\label{def:projections}
	We define the maps $\pi_{\mathbb{Z}^n}\colon \Sigma^* \to \mathbb{Z}^n$ and $\pi_{F_m}\colon \Sigma^* \to F_m$ to be the monoid homomorphism such that
	\[
		(
			\pi_{\mathbb{Z}^n}(\sigma_{i,j}),\,
			\pi_{F_m}(\sigma_{i,j})
		)
		=
		\overline{
			t_{j}
			x_{i}
			(\rho(t_{j}x_{i}))^{-1}
		}
		\in \mathbb{Z}^n \times F_m = H
	\]
	for each $\sigma_{i,j} \in \Sigma$.
	We then see that
	\[
		\overline{w} =
		(
			\pi_{\mathbb{Z}^n}(u),\,
			\pi_{F_m}(u)
		)
		\in \mathbb{Z}^n \times F_m
	\]
	for each $w \in X^*$ with $u = \varphi(w) \in \mathcal{L}_R$.
\end{definition}

From \cref{def:projections}, we have the following observation.

\begin{remark}\label{rmk:g-fun-lang}
	We observe that the generating function $C_R(z)$ counts the words $w \in \mathcal{L}_R$ for which both $\pi_{\mathbb{Z}^n}(w) = \mathbf{0}$ and $\pi_{F_m}(w) = 1$.
\end{remark}

From the definition of the map $\pi_{\mathbb{Z}^n}\colon \Sigma^* \to \mathbb{Z}^n$, we define a set $\mathcal{Z} \subseteq \mathbb{N}^{|\Sigma|}$ as in the following lemma.
The generating function of this set is used in \cref{thm:main} to select the terms of a multivariate power series which correspond to the words $w \in \Sigma^*$ for which $\pi_{\mathbb{Z}^n}(w) = \mathbf{0}$.

\begin{lemma}\label{lem:equation}
	Let $\Sigma$ be the alphabet defined in (\ref{eq:sigma}).
	We choose an order on the letters of $\Sigma$, in particular, we write $\Sigma = \{a_1,a_2,\ldots,a_{|\Sigma|}\}$ where each letter $a_k$ corresponds to a letter of the form $\sigma_{i,j}$.
	Let
	\[
		\mathcal{Z} \coloneqq\left\{
			z \in \mathbb{N}^{|\Sigma|}
		\ \middle|\ 
			\pi_{\mathbb{Z}^n}(a_1)\, z_1 + 
			\pi_{\mathbb{Z}^n}(a_2)\, z_2
			+\cdots+
			\pi_{\mathbb{Z}^n}(a_{|\Sigma|})\, z_{|\Sigma|}
			=
			\mathbf{0}
		\right\}.
	\]
	Equivalently, $z \in \mathcal{Z}$ if and only if there is a word $w\in \Sigma^*$ that, for each $i$, contains exactly $z_i$ instances of the letter $a_i$.
	Then, the generating function
	\[
		f_\mathcal{Z}(y_1,y_2,\ldots,y_{|\Sigma|})
		\coloneqq
		\sum_{z\in \mathcal{Z}} y_1^{z_1} y_2^{z_2} \cdots y_{|\Sigma|}^{z_{|\Sigma|}}
	\]
	is $\mathbb{N}$-rational.
	
\end{lemma}

\begin{proof}
In this proof, for each $z=(z_1,z_2,\ldots,z_{|\Sigma|})\in \mathbb{N}^{|\Sigma|}$, we write
\[
	\mathbf{x}^z
	\coloneqq
	x_1^{z_1}x_2^{z_2}\cdots x_{|\Sigma|}^{z_{|\Sigma|}}.
\]
This notation will simplify the explanations given in this proof.
	
In this proof, we will show that the class of \emph{semilinear} sets have $\mathbb{N}$-rational generating function, and that $\mathcal{Z} \subseteq \mathbb{N}^{|\Sigma|}$ is semilinear.
We begin by recalling the definition of semilinear sets as follows.
	
We say that a subset of $\mathbb{N}^{|\Sigma|}$ is \emph{semilinear} if it can be written as a union of finitely many sets of the form
\[
L(v,\{u_1,u_2,\ldots,u_m\})
=
\{
v+ k_1 u_1 + k_2 u_2 +\cdots+ k_m u_m \in \mathbb{N}^{|\Sigma|}
\mid
k_i \in \mathbb{N}
\}
\]
where $v, u_1,u_2,\ldots,u_m \in \mathbb{N}^{|\Sigma|}$.
Such a set $L(v,\{u_1,\ldots,u_m\})$ is called \emph{linear}.

We see that if the vectors $u_1,u_2,\ldots,u_m$ are linearly independent, with respect to the field $\mathbb{Q}$, then the generating function of the linear set
\[
	L = L(v,\{u_1,u_2,\ldots,u_m\})
\]
is given by
\[
	\sum_{z\in L}
	\mathbf{x}^z
	=
	\mathbf{x}^v
	\prod_{i=1}^m
	\left(
		1-
		\mathbf{x}^{u_i}
	\right)^{-1}
\]
which is $\mathbb{N}$-rational.
Such a linear set is called \emph{simple}.

It is known from Theorem~IV in \cite{Eilenberg1969} that any semilinear set can be written as a disjoint union of finitely many simple linear sets.
Thus, the generating function of a semilinear set can be written as the sum of finitely many $\mathbb{N}$-rational series, and thus is itself $\mathbb{N}$-rational.
All that remains is to show that $\mathcal{Z}$ is semilinear.

We see that $\mathcal{Z}$ is a \textit{Presburger set} as defined in p.~287 of \cite{ginsburg1966}.
In Theorem~1.3 of \cite{ginsburg1966}, it is shown that every Presburger set is semilinear.
We then conclude that $\mathcal{Z}$ is semilinear and thus has an $\mathbb{N}$-rational generating function.
\end{proof}

\Cref{lem:equation} gives us a method of verifying if $\pi_{\mathbb{Z}^n}(w)=\mathbf{0}$ for words $w\in \Sigma^*$, which we describe as follows.
Let $\Sigma = \{a_1,a_2,\ldots,a_{|\Sigma|}\}$ as in \cref{lem:equation}.
Notice that if $w \in \Sigma^*$ and \[z = (\#_{a_1}(w), \#_{a_2}(w),\ldots,\#_{a_{|\Sigma|}}(w) ),\]
where each $\#_{a_i}(w)$ counts the occurrences of the letter $a_i$ in the word $w$.
Then $\pi_{\mathbb{Z}^n}(w) = \mathbf{0}$ if and only if $z \in \mathcal{Z}$.
Thus, we only need a method to verify if $\pi_{F_m}(w) = 1$ for each word $w \in \Sigma^*$.
For this, we provide a deterministic context-free language as follows.

\begin{lemma}\label{lem:det-cf}
	The formal language
	\[
		\mathcal{D}_R\coloneqq\mathcal{L}_R \cap \pi_{F_m}^{-1}(1)\subseteq\Sigma^*
	\]
	is deterministic context-free.
\end{lemma}

\begin{proof}
From \cref{def:L_R}, we see that $\mathcal{L}_R$ is a regular language.
Then, from \cref{lem:intersecion-dcf-regular}, we see that it only remains to be shown that $\pi^{-1}_{F_m}(1) \subseteq \Sigma^*$ is deterministic context-free.

Let $S = \{s_1,s_2,\ldots,s_m, s_1^{-1},s_2^{-1},\ldots,s_m^{-1}\}$ be a finite symmetric generating set for the free group $F_m$.
Let $f\colon \Sigma^* \to S^*$ be a monoid homomorphism defined such that $\pi_{F_m}(\sigma_{i,j}) = \overline{f(\sigma_{i,j})}$ for each $\sigma_{i,j} \in \Sigma$.
Then,
\[
	\pi_{F_m}^{-1}(1)
	=
	f^{-1}
	\left(
		\{
			w \in S^*
		\mid
			\overline{w} = 1
		\}
	\right).
\]
From \cref{lem:dcf-iff-vfree}, we see that $\{w\in S^* \mid \overline{w}=1\}$ is deterministic context-free.
Since $f \colon \Sigma^* \to S^*$ is a monoid homomorphism, it follows from \cref{lem:intersecion-dcf-regular} that $\pi_{F_m}^{-1}(1)$ is also deterministic context-free.
\end{proof}

\section{Main theorem}\label{sec:computing-cogrowth}

The purpose of this section is to prove our main results which follows from some algebraic manipulations on the algebraic generating function of the deterministic context-free language in \cref{lem:det-cf}, and the rational generating function as in \cref{lem:equation}.

\ThmA

\begin{proof}
Let $\Sigma = \{a_1,a_2,\ldots,a_{|\Sigma|}\}$ and $f_\mathcal{Z}(\mathbf{y})$ be the labelling of $\Sigma$, and generating function from \cref{lem:equation}.
That is,
\[
	f_\mathcal{Z}(y_1,y_2,\ldots,y_{|\Sigma|})
	=
	\sum_{z\in \mathcal{Z}}
	y_1^{z_1} y_2^{z_2} \cdots y_{|\Sigma|}^{z_{|\Sigma|}}
\]
where $z\in \mathcal{Z}$ if and only if $\pi_{\mathbb{Z}^n}(w)=\mathbf{0}$ for at least one or equivalently every word $w \in \Sigma^*$ which, for each $i$, contains exactly $z_i$ instances of $a_i$.

Let $\mathcal{D}_R$ be the deterministic context-free language from \cref{lem:det-cf}.
Then, with $\Sigma$ labelled as above, we conclude from \cref{lem:gfun,lem:det-implies-unambiguous}, that the generating function
\[
	g(x_1,\ldots,x_{|\Sigma|},z)
	=
	\sum_{k_1,\ldots,k_{|\Sigma|}}
	c(k_1,\ldots,k_{|\Sigma|})
	x_1^{k_1} \cdots x_{|\Sigma|}^{k_{|\Sigma|}}
	z^{k_1+\cdots + k_{|\Sigma|}}
\]
for $\mathcal{D}_R$
is algebraic.
Notice from the definition of the generating function in \cref{lem:det-implies-unambiguous} that $c(k_1,k_2,\ldots,k_{|\Sigma|})$ is non-zero if and only if there is a word $w \in \mathcal{D}_R $ which contains exactly $k_i$ instances of the letter $a_i$ for each $i$.

From \cref{lem:alg-closure}, we see that
\[
	A(\mathbf{x},\mathbf{y},z)
	=
	g(x_1,x_2,\ldots,x_{|\Sigma|},z)
	f_\mathcal{Z}(y_1,y_2,\ldots,y_{|\Sigma|})
\]
is algebraic.
Notice that the words in the language $\mathcal{D}_R \cap \pi_{\mathbb{Z}^n}^{-1}(\mathbf{0})$ are counted by the coefficients of the terms in $A(\mathbf{x},\mathbf{y},z)$ of the form
\begin{equation}\label{eq:term}
	x_1^{k_1}
	x_2^{k_2}
	\cdots
	x_{|\Sigma|}^{k_{|\Sigma|}}
	\,
	y_1^{k_1}
	y_2^{k_2}
	\cdots
	y_{|\Sigma|}^{k_{|\Sigma|}}
	\,
	z^{\ell}.
\end{equation}
That is, the terms in (\ref{eq:term}) are precisely the ones where the power of $x_i$ and $y_i$ are equal for each $i$.
Moreover, we recall that the single variable generating function of $\mathcal{D}_R\cap\pi_{\mathbb{Z}^n}^{-1}(\mathbf{0})$ is precisely the cogrowth series $C_R(z)$.

At this point of the proof, we can write the cogrowth series $C_R(z)$ by performing $|\Sigma|$ primitive diagonals, and a variable substitution, to the generating function $A(\mathbf{x},\mathbf{y},z)$.
In order to express this as one complete diagonal, we modify this generating function in the following way.

From \cref{lem:alg-closure}, we see that
\[
	B(\mathbf{x},\mathbf{y},z)
	=
	\left(
		\prod_{j=1}^{|\Sigma|}
		\frac{1}{1-x_j y_j}
	\right)
	A(\mathbf{x},\mathbf{y},z),
\]
is algebraic.
We notice then that the coefficients of terms
\[
	x_1^{\ell}
	x_2^{\ell}
	\cdots
	x_{|\Sigma|}^{\ell}
	\,
	y_1^{\ell}
	y_2^{\ell}
	\cdots
	y_{|\Sigma|}^{\ell}
	\,
	z^{\ell}
\]
in $B(\mathbf{x},\mathbf{y},z)$ are sums of coefficients of terms of the form (\ref{eq:term}) in $A(\mathbf{x},\mathbf{y},z)$.
In particular, we see that
\[
	C_R(z) = \mathrm{Diag}(B)(z)
\]
since $\mathrm{Diag}(B)(z)$ is the univariate generating function of $\mathcal{D}_R\cap \pi_{\mathbb{Z}^n}^{-1}(\mathbf{0})$.
Thus, the cogrowth series $C_R(z)$ is the diagonal of an algebraic series.

Our result then follows from \cref{cor:alg-diag-to-rat-diag}.
\end{proof}

One application of \cref{thm:main} is to the Baumslag-Solitar groups $\mathrm{BS}(N,N)$.
In particular, we may combine \cref{thm:main,lem:BS_NN} to obtain the following result.

\CorAa

\begin{proof}
	We first notice that $\mathrm{BS}(0,0) = F_2$ and $\mathrm{BS}(-N,-N) = \mathrm{BS}(N,N)$.
	From \cref{lem:BS_NN}, for each $N\geq 1$, the group $\mathrm{BS}(N,N)$ contains $\mathbb{Z}\times F_N$ as a finite-index subgroup.
	Thus, our result follows from \cref{thm:main}.
\end{proof}

Notice that \cref{thm:main} can also be applied to show that the cogrowth series of any virtually abelian group can be written as the diagonal of a rational series.
In the following, we show a stronger result, i.e., that in the virtually abelian case we can demand that the rational series is $\mathbb{N}$-rational.

\ThmB

\begin{proof}
From \cref{lem:normal-abelian}, we may assume without loss of generality that $G$ contains $H=\mathbb{Z}^n$ as a finite-index normal subgroup.

Let $\mathcal{L}_R \subseteq \Sigma^*$, $\mathcal{Z} \subseteq \mathbb{N}^{|\Sigma|}$ and $\pi_{\mathbb{Z}^n}\colon \Sigma^* \to \mathbb{N}^n$ be defined as in 
\cref{def:L_R},
\cref{lem:equation} and
\cref{def:projections}, respectively.
Moreover, in this proof, we fix an order on the alphabet $
\Sigma = \{
a_1, a_2,\ldots,a_{|\Sigma|}
\}
$.
We now notice that a word $w \in \Sigma^*$ is counted by $C_R(z)$ if and only if both $w \in \mathcal{L}_R$ and $\pi_{\mathbb{Z}^n}(w)=\mathbf{0}$.
In particular, this mean that we do not need to use the deterministic context-free language given in \cref{lem:det-cf} (we only needed this language to deal with non-abelian free subgroups).

From \cref{lem:regular_are_N-rational}, we see that the generating function
\[
f_{\mathcal{L}_R}(\mathbf{x},z)
=
\sum_{k_1,k_2,\ldots,k_{|\Sigma|}}
c(k_1,k_2,\ldots,k_{|\Sigma|})
x_1^{k_1}x_2^{k_2}\cdots x_{|\Sigma|}^{k_{|\Sigma|}}
z^{k_1+k_2+\cdots + k_{|\Sigma|}}
\]
 of $\mathcal{L}_R$ is $\mathbb{N}$-rational.
Let $f_{\mathcal{Z}}(\mathbf{y})$ be the $\mathbb{N}$-rational generating function of $\mathcal{Z}\subseteq\mathbb{N}^{|\Sigma|}$ from \cref{lem:equation}.
We then see that the generating function
\[
	A(\mathbf{x},\mathbf{y},z)
	=
	f_{\mathcal{L}_R}(\mathbf{x},z)
	f_{\mathcal{Z}}(\mathbf{y})
\]
is $\mathbb{N}$-rational.
Moreover, from the definition of the set $\mathcal{Z}$ in \cref{lem:equation}, we see that the words $w = w_1 w_2\cdots w_\ell \in \mathcal{L}_R$ with $\pi_{\mathbb{Z}^n}(w)=\mathbf{0}$ are counted by the coefficients of terms of $A(\mathbf{x},\mathbf{y},z)$ of the form
\begin{equation}\label{eq:term2}
	x_1^{k_1}
	x_2^{k_2}
	\cdots
	x_{|\Sigma|}^{k_{|\Sigma|}}
	\,
	y_1^{k_1}
	y_2^{k_2}
	\cdots
	y_{|\Sigma|}^{k_{|\Sigma|}}
	\,
	z^{\ell}.
\end{equation}
From the definition of $f_{\mathcal{L}_R}(\mathbf{x},z)$, for each term of the form (\ref{eq:term2}) with a non-zero coefficient, we notice that $k_i \leq \ell$ for each $i$.

At this point of the proof, we can write the cogrowth series $C_R(z)$ by performing $|\Sigma|$ primitive diagonals, and a variable substitution, to the generating function $A(\mathbf{x},\mathbf{y},z)$.
In order to express this as one complete diagonal, we modify this generating function in the following way.

We notice that
\[
B(\mathbf{x},\mathbf{y},z)
=
\left(
\prod_{j=1}^{|\Sigma|}
\frac{1}{1-x_j y_j}
\right)
A(\mathbf{x},\mathbf{y},z),
\]
is $\mathbb{N}$-rational, and that the coefficients of terms
\[
x_1^{\ell}
x_2^{\ell}
\cdots
x_{|\Sigma|}^{\ell}
\,
y_1^{\ell}
y_2^{\ell}
\cdots
y_{|\Sigma|}^{\ell}
\,
z^{\ell}
\]
in $B(\mathbf{x},\mathbf{y},z)$ are sums of coefficients of terms of the form (\ref{eq:term2}) in $A(\mathbf{x},\mathbf{y},z)$.
In particular, we see that
$
\mathrm{Diag}(B)(z)
$
is the single variable generating function of $\mathcal{L}_R\cap \pi^{-1}_{\mathbb{Z}^n}(\mathbf{0})$.
We then conclude by writing that the cogrowth series as $C_R(z) = \mathrm{Diag}(B)(z)$, i.e., as the diagonal of an $\mathbb{N}$-rational series.
\end{proof}

\section*{Acknowledgements}

We thank Tatiana Nagnibeda and Corentin Bodart for bringing the questions addressed in this paper to the attention of the author, and for their feedback on early drafts of this paper.
We also thank Murray Elder for suggested improvements to the notation of early versions of the paper.
Additionally, we thank Igor Pak for suggested improvements and citations.

This work is supported by the Swiss Government Excellence Scholarship, and Swiss NSF grant 200020-200400.

\end{document}